\documentclass{article}

\usepackage{amssymb} 
\usepackage{mathrsfs} 
\usepackage{amsmath} 
\usepackage{amsthm}
\usepackage{MnSymbol}

\usepackage[utf8]{inputenc}
\usepackage[english]{babel}

\usepackage{tikz}
\usetikzlibrary{calc}
\usepackage{xargs}

\usepackage[textsize=tiny]{todonotes}

\newtheorem{theorem}{Theorem}[section]
\newtheorem*{theorem*}{Theorem}
\newtheorem{lemma}[theorem]{Lemma}

\newtheorem{fact}[theorem]{Fact}
\newtheorem{definition}[theorem]{Definition}

\newtheorem{question}[theorem]{Question}
\newtheorem{observation}[theorem]{Observation}
\newtheorem{convention}[theorem]{Convention}

\newcounter{claim}[theorem]

\newenvironment{claim}[1][]
{\refstepcounter{claim}\par\textbf{Claim. #1}\rmfamily}
{\par}

\newcommandx{\Ynm}[2][1=n, 2=m]{\mathscr{Y}_{#1,#2}}
\newcommandx{\Znm}[2][1=n, 2=m]{\mathscr{Z}_{#1,#2}}
\newcommandx{\Fnm}[2][1=n, 2=m]{\mathcal{F}_{#1,#2}}
\newcommandx{\hnm}[2][1=n, 2=m]{\eta_{#1,#2}}

\newcommandx{\uz}[2][1=n, 2=m]{u^0_{#1,#2}}
\newcommandx{\uo}[2][1=n, 2=m]{u^1_{#1,#2}}

\newcommandx{\dnm}[2][1=n, 2=m]{d_{#1,#2}}
\newcommandx{\dz}[2][1=n, 2=m]{d^0_{#1,#2}}
\newcommandx{\doo}[2][1=n, 2=m]{d^1_{#1,#2}}

\newcommand{\diam}{\operatorname{diam}}
\newcommand{\ecc}{\operatorname{ecc}}
\newcommand{\piv}{\operatorname{Piv}}

\newcommand{\ps}{\sigma}
\newcommand{\Icwidth}{\gamma}

\begin{document}
\title{Flip Graphs, Yoke Graphs and Diameter}
\author{Roy H. Jennings
	\thanks{Department of Mathematics, Bar-Ilan University, Israel. royh.jennings@gmail.com.}
	\thanks{Supported by the Israel Science Foundation grant no. 1970/18.}}
\maketitle
\begin{abstract}
	In this paper we introduce Yoke graphs, a family of flip graphs that generalizes several previously studied families of graphs: colored triangle free triangulations, arc permutations and caterpillars.
	Our main result is the computation of the diameter of an arbitrary Yoke graph.
\end{abstract}

\section{Introduction}
Over recent decades there has been an increasing interest in graphs on combinatorial objects in which the adjacency relation reflects a local change, for example, flip graphs (see \cite{bose, fabila, jennings_gpm, parlier, pournin, tarjan}).
In this paper we introduce Yoke graphs, a family of flip graphs that generalizes previously studied families of graphs: colored triangle free triangulations \cite{TFT1} (CTFT), arc permutations \cite{elizalde} and geometric caterpillars \cite{yuval}.

The flip graph of triangulations of a convex polygon \cite{tarjan} inspired the definition of a few flip graphs on subsets of triangulations.
One such graph is the CTFT graph.
This graph is closely related to a distinguished lower interval in the weak order on the affine Weyl group $\widetilde{C}_n$. 
The diameter of this flip graph was calculated using lattice properties of the order, see \cite{TFT1}.

An arc permutation in the symmetric group $S_n$ is one in which every prefix (and suffix) forms an interval in $\mathbb{Z}_n$.
The flip graph of arc permutations is the subgraph induced by the set of arc permutations in the Cayley graph associated with $(S_n,S)$, where $S$ is the generating set of $S_n$ consisting of simple reflections.
The diameter of the graph of arc permutations was computed using similarities between the graph and the dominance order on $\mathbb{Z}^n$, see \cite{elizalde}.

To define a geometric caterpillar, start with the complete graph $K_n$ whose vertices are labeled by $\mathbb{Z}_n$. 
Embed $K_n$ in the plane such that its vertices form a convex polygon and its edges are straight line segments.
Denote this geometric graph by $\hat{K}_n$.
A geometric caterpillar of order $n$ is a non-crossing spanning tree of $\hat{K}_n$, such that the vertices of its spine form an interval in $\mathbb{Z}_n$.
For a formal definition see \cite{keller}.
Caterpillars are also called fishbones or combs and were studied by Keller, Perles, Sagan, Wachs and others in various contexts, see, e.g., \cite{perles, martin}.
In the flip graph of geometric caterpillars, two caterpillars are adjacent if one is obtained from the other by moving an edge incident with a leaf along the spine.
The problem of the diameter of the flip graph of geometric caterpillars is open.

The generalization (by Yoke graphs) of these three flip graphs is motivated by their surprising similarities in terms of algebraic, combinatorial and metric properties.
In particular, they carry similar group actions, are intimately related to posets and have similar diameter formulas. 

Our main result, Theorem \ref{thm:yoke_graph_diam}, is the computation of the diameter of an arbitrary Yoke graph.
Since Yoke graphs generalize the flip graphs described above, this theorem provides a unified proof for the previous results.
Our approach is different from the ones in \cite{TFT1} and \cite{elizalde}.
At the heart of the proof lies the idea of transforming a diameter evaluation into an eccentricity problem.

\section{Yoke Graphs}\label{sec:YokeGraphs}

\begin{definition}\label{def:yoke_graph}
	Let $n\geq 1$ and $m\geq 0$ be two integers.
	Denote the subset $\{0,1\}\subseteq \mathbb{Z}$ by $P_2$.
	The \textbf{Yoke graph} $\Ynm$ is a graph with vertices corresponding to all $u=(u_0\dots,u_{m+1})\in\mathbb{Z}_n\times P_2^m\times \mathbb{Z}_n$ such that $\sum_{i=0}^{m+1}u_i\equiv 0\ (\mathrm{mod}\ n).$
	Two vertices $u$ and $v$ are adjacent in $\Ynm$, denoted $u\sim v$, if there exists $0\leq i\leq m$ such that $u_j=v_j$ for every $j\notin\{i,i+1\}$ and one of the following two cases holds: either $u_i=v_i+1$ and $u_{i+1}=v_{i+1}-1$, or $u_i=v_i-1$ and $u_{i+1}=v_{i+1}+1$. 
\end{definition}
\begin{convention}\label{conv:cosets_are_integers}
Throughout this paper, elements (cosets) in a quotient group $\mathbb{Z}_n$ of $\mathbb{Z}$ are identified with their smallest non-negative representative in $\mathbb{Z}$.
\end{convention}
According to convention \ref{conv:cosets_are_integers}, the sum $\sum_{i=0}^{m+1}u_i$ in Definition \ref{def:yoke_graph} is a non-negative integer. 
Denote the vertex $(0,\dots,0)\in\Ynm$ by $0$.
For an example of a Yoke graph see $\Ynm[3][3]$ in Figure \ref{fig:Y33}.

\begin{figure}[hbt]
	\centering
	\begin{tikzpicture}[scale=0.4]
	\tikzstyle{every node}=[draw,circle,fill=white,minimum size=4pt,inner sep=0pt]
	
	\draw (0,0) node (01110) [label=left:\scriptsize(01110)] {}
	-- ++(-120:2.0cm) node (01101) [label=left:\scriptsize(01101)] {}
	-- ++(-120:4.0cm) node (01011) [label=right:\scriptsize(01011)] {}
	-- ++(-120:4.0cm) node (00111) [label=left:\scriptsize(00111)] {}
	-- ++(-120:2.0cm) node (21111) [label=left:\scriptsize(21111)] {}
	
	-- ++(0:2.0cm) node (21102) [label={[shift={(-90:0.8)}]:\scriptsize(21102)}] {}
	-- ++(0:4.0cm) node (21012) [label={[shift={(90:-0.3)}]:\scriptsize(21012)}] {}
	-- ++(0:4.0cm) node (20112) [label={[shift={(-90:0.8)}]:\scriptsize(20112)}] {}
	-- ++(0:2.0cm) node (11112) [label=right:\scriptsize(11112)] {}
	
	-- ++(120:2.0cm) node (11100) [label=right:\scriptsize(11100)] {}
	-- ++(120:4.0cm) node (11010) [label=left:\scriptsize(11010)] {}
	-- ++(120:4.0cm) node (10110) [label=right:\scriptsize(10110)] {}
	-- (01110) {};
	
	\draw (01101) -- ++(-60:2.0cm) node (10101) [label={[shift={(-90:0.8)}]:\scriptsize(10101)}] {}
	-- ++(0:4.0cm) node (11001) [label={[shift={(90:-0.3)}]:\scriptsize(11001)}] {}
	-- ++(0:2.0cm) node (20001) [label=right:\scriptsize(20001)] {}
	
	-- ++(-120:2.0cm) node (20010) [label=right:\scriptsize(20010)] {}
	-- ++(-120:4.0cm) node (20100) [label=left:\scriptsize(20100)] {}
	-- ++(-120:4.0cm) node (21000) [label=right:\scriptsize(21000)] {}
	-- ++(-120:2.0cm) node (00000) [label=right:\scriptsize 0] {}
	
	-- ++(120:2.0cm) node (00012) [label=left:\scriptsize(00012)] {}
	-- ++(120:4.0cm) node (00102) [label=right:\scriptsize(00102)] {}
	-- ++(120:4.0cm) node (01002) [label=left:\scriptsize(01002)] {}
	-- ++(120:2.0cm) node (10002) [label=left:\scriptsize(10002)] {}
	
	-- ++(0:2.0cm) node (10011) [label={[shift={(90:-0.3)}]:\scriptsize(10011)}] {}
	-- (10101) {};
	
	\draw (10101) -- (10110);
	\draw (11001) -- (11010) -- (20010);
	\draw (11100) -- (20100) -- (20112);
	\draw (21000) -- (21012) -- (00012);
	\draw (21102) -- (00102) -- (00111);
	\draw (01002) -- (01011) -- (10011);
	
	\end{tikzpicture}
	\caption{$\Ynm[3][3]$} \label{fig:Y33}
\end{figure}
The name "Yoke graph" is derived from the shoulder yoke, a tool that can be used to carry two buckets.
Based on this analogy, we will sometimes refer to the entries $u_0$ and $u_{m+1}$ of a vertex $u$ in $\Ynm$ as \textbf{buckets}.

As noted in the introduction, important special cases of yoke graphs include the following. The graph $\Ynm[n][n-4]$ is isomorphic to the CTFT graph.
The map defined in \cite[Definition 2.8]{TFT1} (which was used to calculate the cardinality of the CTFT graph) induces an isomorphism between this graph and $\Ynm[n][n-4]$.
The graph $\Ynm[n][n-2]$ is isomorphic to the flip graph of arc permutation. It can be verified that the encoding $\psi$, defined in \cite[Section 6.2]{elizalde}, induces an isomorphism between the graph of arc permutations and $\Ynm[n][n-2]$.
Similarly, it can be shown that $\Ynm[n][n-3]$ is isomorphic to the flip graph of geometric caterpillars.

Yoke graphs $\Ynm$ are Schreier graphs of the affine Weyl group of type $\tilde{C}_m$ whenever $m>1$.
This fact is naturally extended from previously known results.
For example, by \cite[Proposition 3.2]{TFT1}, the CTFT graph $\Ynm[n][n-4]$ is a Schreier graph of the affine Weyl group of type $\tilde{C}_{n-4}$ for $n>5$.
Also, by \cite[Corollary 10.4]{elizalde}, the arc permutations graph $\Ynm[n][n-2]$ is a Schreier graph of the affine Weyl group of type $\tilde{C}_{n-2}$ for $n>3$.
For more details see \cite{jennings}.

The main result of this paper is the following. \begin{theorem}\label{thm:yoke_graph_diam}
	If $m\leq n$, then $diam(\Ynm) = \lfloor\frac{n(m+1)}{2}\rfloor$.
	
	If $1=n\leq m$, then $\diam(\Ynm)=\binom{\lceil \frac{m}{2}\rceil + 1}{2} + \binom{\lfloor \frac{m}{2}\rfloor + 1}{2}$; Otherwise, $2\leq n\leq m$ and
	\begin{enumerate}
		\item if $2\divides (m-n)$ or $n\leq\lceil\frac{m+1}{2}\rceil$, then 
		$$\diam(\Ynm) = \dz = \binom{\lfloor\frac{m+n}{2}\rfloor+1}{2} + \binom{\lceil\frac{m-n}{2}\rceil+1}{2};$$
		\item otherwise, $$\diam(\Ynm) = \dz + n-\lceil\frac{m+1}{2}\rceil.$$
	\end{enumerate}
\end{theorem}

The proof of Theorem \ref{thm:yoke_graph_diam} appears at the end of Subsection \ref{subsec:n_leq_m}.
At the heart of the proof lies the idea to convert the problem of computation of diameter to that of computation of eccentricity.
Specifically, in Section \ref{sec:dYokeGraphs}, we introduce dYoke graphs $\Znm$, which are closely related to Yoke graphs. 
In Section \ref{sec:diam_ecc}, we show that the diameter of $\Ynm$ is equal to the eccentricity of $0$ in $\Znm$.

\section{dYoke Graphs}\label{sec:dYokeGraphs}
As noted in Section \ref{sec:YokeGraphs}, at the heart of the calculation of the diameter of $\Ynm$ lies the idea of converting the diameter problem of one graph into an eccentricity problem in another graph.
To this end, we introduce a family of graphs.

\begin{definition}
	Let $n\geq 1$ and $m\geq 0$ be two integers. 
	Denote the subset $\{-1,0,1\}\subseteq \mathbb{Z}$ by $P_3$.
	The \textbf{dYoke graph} $\Znm$ is a graph with vertices corresponding to all $u=(u_0\dots,u_{m+1})\in\mathbb{Z}_n\times P_3^m\times \mathbb{Z}_n$ such that $\sum_{i=0}^{m+1}u_i\equiv 0\ (\mathrm{mod}\ n).$ 
	Two vertices $u$ and $v$ are adjacent in $\Znm$ if there exists $0\leq i\leq m$ such that $u_j=v_j$ for every $j\notin\{i,i+1\}$ and one of the following two cases holds: either $u_i=v_i+1$ and $u_{i+1}=v_{i+1}-1$, or $u_i=v_i-1$ and $u_{i+1}=v_{i+1}+1$.
\end{definition}
Note that vertices in Yoke graphs and dYoke graphs are determined by their first (or last) $m+1$ entries, since $\sum_{i=0}^{m+1}u_i\equiv 0(\bmod n)$.
We denote the vertex $(0,\dots,0)\in\Znm$ by $0$.
We will sometimes refer to the entries $u_0$ and $u_{m+1}$ of a vertex $u$ in $\Znm$ as \textbf{buckets}, similarly to Yoke graphs.
In the first case of the adjacency relation, where $u_i=v_i+1$ and $u_{i+1}=v_{i+1}-1$ for some $0\leq i\leq m$, we say that $u$ is obtained from $v$ by \textbf{shifting} a unit from entry $i+1$ to the \textbf{left}, and write $u=\overleftarrow{s}_{i}(v)$.
In the second case, where $u_i=v_i-1$ and $u_{i+1}=v_{i+1}+1$, we say that $u$ is obtained from $v$ by \textbf{shifting} a unit from entry $i$ to the \textbf{right}, and write $u=\overrightarrow{s}_{i}(v)$.
When $v_0=0$ ($v_{m+1}=0$) we say that the left (right) bucket is \textbf{empty}.

For example, if $v=(3,0,-1,1,2)\in\Znm[5][3]$, then $\overleftarrow{s}_2(v)=(3,0,0,0,2)$ and $\overrightarrow{s}_1(v)=(3,-1,0,1,2)$.
If $i<j$, we say that $v_j$ is an \textbf{entry to the right} of $v_i$ and that $v_i$ is an \textbf{entry to the left} of $v_j$. 
For convenience, we write $s_i$ to indicate a unit shift between the entries indexed by $i$ and $i+1$ without specifying its direction.

If $1\leq k\leq m$ and $v_k=-1$, then a unit shift from entry $k$ is not possible.
Similarly, if $v_k=1$, then a unit shift to entry $k$ not possible. 
However, for every $0\leq i\leq m$, $\overleftarrow{s}_i$ and $\overrightarrow{s}_i$ still induce functions from $\Znm$ to $\Znm$, in which if a unit shift $s_i(v)$ is not possible for some $v\in\Znm$, then $v$ is a fixed point of $s_i$.
Denote the set $\{\overleftarrow{s}_i,\overrightarrow{s}_i:0\leq i\leq m\}$ by $\Fnm$.
If $m=0$, then $\overleftarrow{s}_0$ and $\overrightarrow{s}_0$ are inverses of each other and $\Fnm[n][0]$ generates the cyclic group $\mathbb{Z}_n$. 
If $m>0$, then none of the functions in $\Fnm$ are bijective and $\Fnm$ generates a semigroup.

A \textbf{word $w$ in the letters} $\Fnm$ is a sequence $f_d\cdots f_1$ where $f_1,\dots,f_d\in\Fnm$.
Let $P$ be a path $(v=v^0\sim\dots \sim v^d=u)$ in $\Znm$. 
In this paper, when we write $u\sim v$, it implies that $(u,v)$ is an edge in the graph (and that $u\neq v$, since the graphs in this paper have no loops).
Clearly, there is a unique word $f_d\cdots f_1$ such that $f_t(v^{t-1})=v^t$ for every $1\leq t\leq d$ or, equivalently,  $f_t\cdots f_1(v)=v^t$ for every $1\leq t\leq d$.
We say that $f_d\cdots f_1$ is the \textbf{word corresponding to the path} $P$ from $v$ to $u$. 

Let $f_d\cdots f_1$ be the word corresponding to a path from $v$ to $u$ in $\Znm$.
If $f_d\cdots f_{t+1}f_t\cdots f_1(v) = f_d\cdots f_{t}f_{t+1}\cdots f_1(v)$ (changing the order of $f_t$ and $f_{t+1}$), then we say that $f_t$ and $f_{t+1}$ \textbf{relatively commute in} (the path) $f_d\cdots f_1(v)$.
Note that the fact that $f_t$ and $f_{t+1}$ relatively commute in $f_d\cdots f_1(v)$, does not imply that $f_t$ and $f_{t+1}$ commute (as elements in the semigroup generated by $\Fnm$). 
For example, if $f_d\cdots f_1$ corresponds to a path from $v$ to $u$ such that $f_t=\overrightarrow{s}_i$ and $f_{t+1}=\overleftarrow{s}_{i+1}$ for some $1\leq t \leq d-1$ and $0\leq i\leq m-1$, then $f_t$ and $f_{t+1}$ relatively commute in $f_d\cdots f_1(v)$ (since a path has no repetitions and both $f_t$ and $f_{t+1}$ indeed shift a unit) but do not commute (in the semigroup).
If, however, $f_t$ and $f_{t+1}$ commute, then they also relatively commute in $f_d\cdots f_1(v)$.
Clearly, two distinct $s_i$ and $s_j$ in $\Fnm$ commute if and only if $|i-j|>1$.

\begin{lemma}[\sc Shift Direction Lemma] \label{lem:shift_direction_lemma_Znm}
	Let $f_d\cdots f_1$ be the word corresponding to a geodesic in $\Znm$.
	For every $i$ with $0\leq i\leq m$, all of the instances of $s_i$ in $f_d\cdots f_1$ shift in the same direction.
\end{lemma}

\begin{proof}
	Let $v$ and $u$ be two vertices in $\Znm$.
	If $m=0$, then a word corresponding to a geodesic from $v$ to $u$ is of the form $f^k$ where $f\in\Fnm[n][0]$ and $0\leq k\leq \lfloor\frac{n}{2}\rfloor$.
	Otherwise $m>0$. 
	Assume to the contrary that there exist words corresponding to geodesics from $v$ to $u$ not satisfying the lemma.
	Denote the set of such words by $\mathcal{W}$.
	For every word $f_d\cdots f_1$ in $\mathcal{W}$, there exist $i, t_1, t_2$ with $0\leq i\leq m$ and $1\leq t_1, t_2\leq d$ such that $f_{t_1}=\overrightarrow{s}_{i}$ and $f_{t_2}=\overleftarrow{s}_{i}$.
	Let $w=f_d\cdots f_1$ be a word in $\mathcal{W}$ in which $\min\{|t_2-t_1|: f_{t_1}=\overrightarrow{s}_{i},f_{t_2}=\overleftarrow{s}_{i}\}$ is minimal. 
	Let $1\leq t_1< t_2\leq d$ be a pair of indices obtaining this minimum and assume without loss of generality that $f_{t_1}=\overrightarrow{s}_i$ and $f_{t_2}=\overleftarrow{s}_i$ where $0\leq i\leq m$.

	If $t_2=t_1+1$, then the word obtained by deleting both $f_{t_1}$ and $f_{t_2}$ from $f_d\cdots f_1$ is a word corresponding to a path from $v$ to $u$, contradicting the minimality of $d$.
	Therefore, we can assume that $t_2-t_1>1$.
	Note that $f_{t_1}$ does not relatively commute with $f_{t_1+1}$ by the minimality of $t_2-t_1$ in the choice of $w$.
	Therefore we can assume that $f_{t_1+1}\in\{\overrightarrow{s}_{i-1}, \overrightarrow{s}_{i+1}\}$ (the only two elements in $\Fnm$ that do not necessarily relatively commute with $f_{t_1}$).
	
	If $f_{t_1+1}=\overrightarrow{s}_{i-1}$, then we can assume that $(f_{t_1+1}\cdots f_1(v))_i=1$, otherwise $f_{t_1}$ and $f_{t_1+1}$ relatively commute in $f_d\cdots f_1(v)$. 
	Moreover, there must be some $k$ with $t_1+1<k<t_2$ such that $f_k=\overleftarrow{s}_{i-1}$, since $f_{t_2}=\overleftarrow{s}_i$.
	Similarly, if $f_{t_1+1}=\overrightarrow{s}_{i+1}$, then we can assume that $(f_{t_1+1}\cdots f_1(v))_{i+1}=-1$ and there must be $f_k=\overleftarrow{s}_{i+1}$ for some $k$ with $t_1+1<k<t_2$.
	Both cases contradicting the minimality of $t_2-t_1$.
\end{proof}

\section{From Diameter to Eccentricity}\label{sec:diam_ecc}
Recall that the eccentricity of a vertex $v$ in a graph is the maximum distance between $v$ and any other vertex.
In this section, we show (Theorem \ref{thm:ecc_eq_diam}) that the diameter of $\Ynm$ is equal to the eccentricity of $0$ in $\Znm$.
Every Yoke graph $\Ynm$ is naturally embedded in the dYoke graph $\Znm$ as an induced subgraph on the vertices with no negative entries. 
Note that for every two vertices $v,u\in\Ynm$, the difference $v-u=(v_0-u_0,\dots,v_{m+1}-u_{m+1})$ is in $\Znm$.

\begin{definition}
	For every $u\in\Ynm$ let $\varphi_u:\Ynm\rightarrow \Znm$ be the map (on vertices) defined by $\varphi_u(v)=v-u$. 
\end{definition}

\begin{observation}\label{obs:embedding_Ynm_in_Znm}
	For every $u\in\Ynm$, $\varphi_u$ is a graph isomorphism between $\Ynm$ and the subgraph induced by $\varphi_u(\Ynm)$ (in $\Znm$).
\end{observation}

The following lemma is essential for the proof of Lemma \ref{lem:dist_preserved_under_phi}.
We start with a given $z\in\Znm$ such that $z_i=0$ for some $1\leq i\leq m$, and a geodesic $P$ between $z$ and $0$.
We show that $P$ can be transformed into a geodesic $P'$ in which the $i$th entry is either non-negative or non-positive along the path. We can do this without changing the sets of values of other entries.

\begin{lemma}\label{lem:handle_zeros_in_geodesics_znm}
	Let $z\in\Znm$ such that $z_i=0$ for some $i$ with $1\leq i\leq m$ and let $P=(z=x^0\sim x^1\sim\dots\sim x^d=0)$ be a geodesic between $z$ and $0$.
	Then a geodesic $P'=(z=y^0\sim y^1\sim\dots\sim y^d=0)$ exists such that:
	{ 
		\begin{enumerate}
			\item either $y^t_i\leq 0$ for every $0\leq t\leq d$ or $y^t_i\geq 0$ for every $0\leq t\leq d$ ($P'$ can be constructed either way);
			\item \label{asdf}$\{y^t_j:0\leq t\leq d\}=\{x^t_j:0\leq t\leq d\}$ for every $j\neq i$ where $0\leq j\leq m+1$.
	\end{enumerate} }
\end{lemma}
\begin{proof}
	We prove the existence of $P'$ such that $y^t_i\leq 0$ for every $0\leq t\leq d$.
	The proof for the case $y^t_i\geq 0$ follows by symmetric arguments.
	Let $Q=(w^0\sim w^1\sim\dots\sim w^l)$ be a path in $\Znm$. Define $O_i(Q)=|\{0\leq t\leq l:w^t_i=1\}|$.
	If $O_i(P)=0$, then set $P'=P$.
	Otherwise, assume that $O_i(P)>0$.
	We show that there exists a geodesic $Q=(z=w^0\sim w^1\sim\dots\sim w^d=0)$ with the following properties:
	{ 
		\renewcommand\labelenumi{(\theenumi)}
		\begin{enumerate}
			\item $O_i(Q)<O_i(P)$.
			\item $\{w^t_j:0\leq t\leq d\}=\{x^t_j:0\leq t\leq d\}$ for every $j\neq i$ where $0\leq j\leq m+1$.
		\end{enumerate} }
	Note that if $O_i(Q)>0$, then the process can be repeated implying the existence of $P'$, as required.
	
	Let $f_d\cdots f_1$ be the word corresponding to $P$.
	Note that by the Shift Direction Lemma \ref{lem:shift_direction_lemma_Znm}, either both $\overrightarrow{s}_{i-1}$ and $\overrightarrow{s}_i$ or both $\overleftarrow{s}_{i-1}$ and $\overleftarrow{s}_i$ appear in $f_d\cdots f_1$, since $x^0_i=0$, $x^d_i=0$ and $O_i(P)>0$. 
	Assume without loss of generality that both $\overrightarrow{s}_{i-1}$ and $\overrightarrow{s}_i$ appear in $f_d\cdots f_1$.
	
	Let $t_1$ be the minimal index such that $x^{t_1}_i=1$ and therefore $f_{t_1}=\overrightarrow{s}_{i-1}$.
	Let $t_2$ be the minimal index such that $t_1<t_2$ and $f_{t_2}=\overrightarrow{s}_{i}$ (such $t_2$ exists since $x^d_i=0$).
	If $t_2-t_1=1$, then $f_{t_1}$ and $f_{t_2}$ relatively commute in $f_d\cdots f_1(z)$, and the path $Q$ obtained by interchanging $f_{t_1}$ and $f_{t_2}$ satisfies properties (1) and (2) as required.
	In the rest of this proof, we assume that $t_2-t_1>1$.
	
	We can assume that $f_d\cdots f_1$ is in the form 
	\begin{equation*} 
	f_d\dots f_{t_2}w_2w_1f_{t_1}\dots f_1  \tag{$*$}
	\end{equation*}
	where $w_1$ and $w_2$ are words (at least one of which is nonempty) such that every letter in $w_1$ shifts a unit between entries to the right of $i$, and every letter in $w_2$ shifts a unit between entries to the left of $i$.
	Indeed, if there exists $t$ with $t_1<t<t+1<t_2$ such that $f_{t}$ shifts a unit between entries to the left of $i$ and $f_{t+1}$ shifts a unit between entries to the right of $i$, then $f_t$ and $f_{t+1}$ commute and	the word obtained by interchanging $f_t$ and $f_{t+1}$ corresponds to a path between $z$ and $0$ which satisfies property (2), and which does not change $O_i(P)$.
	By repeatedly interchanging such pairs, we obtain a word in the form $(*)$, since $f_{t}\notin \{\overrightarrow{s}_{i-1}, \overrightarrow{s}_{i}\}$ for every $t$ with $t_1<t<t_2$.
	
	Note that $f_{t_1}$ commutes with every letter in $w_1$ and $f_{t_2}$ commutes with every letter in $w_2$. 
	Therefore, $f_d\dots w_2f_{t_1}f_{t_2}w_1\dots f_1$ corresponds to a path $Q$ between $z$ and $0$ satisfying properties (1) and (2), similarly to the case $t_2 - t_1 = 1$.
\end{proof}

\begin{lemma}\label{lem:dist_preserved_under_phi}
	Let $v,u\in\Ynm$. Then $d_{\Ynm}(v,u)=d_{\Znm}(v-u,0)$, where $d_G(v,u)$ is the distance between $v$ and $u$ in the graph $G$.
\end{lemma}
\begin{proof}
	By Observation \ref{obs:embedding_Ynm_in_Znm}, $\Ynm$ is isomorphic to the induced subgraph $\varphi_u(\Ynm)$ of $\Znm$, implying that $d_{\Ynm}(v,u)\geq d_{\Znm}(\varphi_u(v),\varphi_u(u))=d_{\Znm}(v-u,0)$.
	
	Let $P=(v-u=x^0\sim x^1\sim\dots\sim x^d=0)$ be a geodesic between $v-u$ and $0$ in $\Znm$.
	Note that if $x^0_i\leq 0$ for some $1\leq i\leq m$, then, by Lemma \ref{lem:handle_zeros_in_geodesics_znm}, there exists a geodesic $P'=(v-u=y^0\sim y^1\sim\dots\sim y^d=0)$ such that $y^t_i\leq 0$ for every $0\leq t\leq d$ (the lemma is applied to $(x^k\sim x^{k+1}\sim\dots\sim x^d=0)$, where $k$ is the first index such that $x^k_i=0$).
	Similarly, if $x^0_i\geq 0$ for some $1\leq i\leq m$, then there exists a geodesic $P'=(v-u=y^0\sim y^1\sim\dots\sim y^d=0)$ such that $y^t_i\geq 0$ for every $0\leq t\leq d$.
	
	Therefore, by its second property, Lemma \ref{lem:handle_zeros_in_geodesics_znm} can be applied iteratively to all $1\leq i\leq m$, to construct a geodesic
	$P'=(v-u=y^0\sim y^1\sim\dots\sim y^d=0)$ from the geodesic $P$ that satisfies the following conditions for every $1\leq i\leq m$:
	\begin{enumerate}
		\item If $u_i=1$ (implying that $(v-u)_i\leq 0$), then $y^t_i\leq 0$ for every $0\leq t\leq d$.
		\item If $u_i=0$ (implying that $(v-u)_i\geq 0$), then $y^t_i\geq 0$ for every $0\leq t\leq d$.
	\end{enumerate}
	Clearly, $(v=y^0+u\sim y^1+u\sim\dots\sim y^d+u=u)$ is a path between $v$ and $u$ in $\Ynm$. Therefore $d_{\Ynm}(v,u)\leq d_{\Znm}(v-u,0)$.
\end{proof}

\begin{theorem}\label{thm:ecc_eq_diam}
	$\diam(\Ynm) = \ecc_{\Znm}(0)$.
\end{theorem}
\begin{proof}
	By Lemma \ref{lem:dist_preserved_under_phi}, $\diam(\Ynm) \leq \ecc_{\Znm}(0)$, since $d_{\Ynm}(v,u)=d_{\Znm}(v-u,0)$ for each pair of antipodes (vertices at maximal distance) $v$ and $u$ in $\Ynm$.
	On the other hand, for every $z\in\Znm$ there exist $v, u\in\Ynm$ such that $v-u=z$.
	Let $z\in\Znm$.
	Construct $v,u\in\Ynm$ as follows; for example,
	\begin{enumerate}
		\item $v_0 = z_0$ and $u_0 = 0$.
		\item For every $i$ with $1\leq i\leq m$, $v_i=\max\{0,z_i\}$ and $u_i=\max\{0,-z_i\}$.
	\end{enumerate}
	Then, by Lemma \ref{lem:dist_preserved_under_phi}, $d_{\Znm}(z,0) = d_{\Znm}(v-u,0) = d_{\Ynm}(v,u) \leq \diam(\Ynm)$, and therefore $\ecc_{\Znm}(0) \leq \diam(\Ynm)$. This proves equality.
\end{proof}

\section{The Eccentricity of $0$ in $\Znm$}
In this section we compute the eccentricity of $0$ in $\Znm$. 
We prove that it is equal to the value of the diameter of $\Ynm$ as it appears in Theorem \ref{thm:yoke_graph_diam}.
\begin{observation}\label{obs:corner_case_meqz}
If $m=0$, then $\Znm[n][0]$ is merely a cycle graph on $n$ vertices ($n$ vertices connected in a closed chain).
Therefore, $\ecc_{\Znm[n][0]}(0)=\lfloor\frac{n}{2}\rfloor=\lfloor\frac{n(m+1)}{2}\rfloor$, in accordance with Theorem \ref{thm:yoke_graph_diam} for the case $m\leq n$.
In the rest of this paper we consider only dYoke graphs in which $m>0$.
\end{observation}
This section is composed of two subsections.
In Subsection \ref{subsec:pivot_paths}, we first introduce pivot paths and some related definitions. We then use them to compute the eccentricity of $0$ when $m\leq n$.
In Subsection \ref{subsec:n_leq_m} we deal with the case $n\leq m$.

\subsection{Pivot Paths and the Case $m\leq n$}\label{subsec:pivot_paths}

\begin{definition}[\sc Pivot]
	A \textbf{pivot of a vertex} $v\in\Znm$ is an integer $-1\leq p\leq m+1$ such that $\sum_{i=0}^{p}v_i$ is divisible by $n$ ($\sum_{i=j}^{k}v_i$ is defined as $0$ whenever $k<j$).
	We call $-1$ and $m+1$ the \textbf{outer pivots} of $v$ (they are pivots of every $v\in\Znm$), and every other pivot, if it exists, is called an \textbf{inner pivot}.
	Denote the set of pivots of $v$ by $\piv(v)$.
\end{definition}

For example, in $\Znm[3][8]$, $\piv((0,1,-1,0,1,1,-1,-1,1,2))=\{-1,0,2,3,7,9\}$.
\begin{definition}[\sc Wall]
	Let $P$ be a path from $v$ to $0$ in $\Znm$ and let $f_d\cdots f_1$ be the word corresponding to $P$.
	An \textbf{inner wall} of $P$ is an integer $0\leq p\leq m$ such that $s_p$ does not appear in $f_d\cdots f_1$.
	If $\overleftarrow{s}_0$ does not appear in $f_d\cdots f_1$, then $-1$ is the \textbf{left outer wall} of $P$. 
	Similarly, if $\overrightarrow{s}_m$ does not appear in $f_d\cdots f_1$, then $m+1$ is the \textbf{right outer wall} of $P$.
	We say that $-1\leq p\leq m+1$ is a \textbf{wall} of $P$ if it is either an inner wall or an outer wall of $P$.
\end{definition}

\begin{definition}[\sc Pivot Path]
	We say that a path $P$ from $v$ to $0$ in $\Znm$ is a \textbf{$p$-pivot path} (or simply a pivot path) of $v$, if $P$ is a shortest path with a wall $p$ (not necessarily a geodesic) from $v$ to $0$.
	Denote the length of a $p$-pivot path of $v$ by $\ps_p(v)$.
\end{definition}

Clearly, $p$-pivot paths of $v$ exist if and only if $p\in\piv(v)$.

\begin{observation}\label{obs:trivial_path}
	Every inner pivot $p$ of $v\in\Znm$ induces a path from $v$ to $0$ of length $\sum_{i=1}^{p}i|v_i| + \sum_{i=p+1}^{m}(m+1-i)|v_i|$. 
	It is a path in which the entries $\{v_1,\dots,v_p\}$ and $\{v_{p+1}, \dots, v_m\}$ of $v$ are handled one by one, left to right and right to left, respectively; each $v_i$, in turn, can be set to $0$ by $i$ unit shifts from (if $v_i=-1$) or to (if $v_i=1$) the bucket on the same side of $p$ as $v_i$.
	Therefore
	$$\ps_p(v)\leq \sum_{i=1}^{p}i|v_i| + \sum_{i=p+1}^{m}(m+1-i)|v_i|\leq \sum_{i=1}^{p}i + \sum_{i=1}^{m-p}i.$$
\end{observation}

The proof of the Shift Direction Lemma \ref{lem:shift_direction_lemma_Znm} can be slightly modified so that the lemma applies to pivot paths.

\begin{lemma}[\sc Pivot Shift Direction Lemma]\label{lem:pivot_shift_direction_lemma}
	Let $f_d\cdots f_1$ be the word corresponding to a pivot path in $\Znm$.
	For every $i$ with $0\leq i\leq m$, all of the instances of $s_i$ in $f_d\cdots f_1$ shift in the same direction.
\end{lemma}

\begin{lemma}\label{lem:geodesic_is_a_pivot_path}
	Every geodesic between every vertex and $0$ in $\Znm$ is a pivot path.
\end{lemma}
\begin{proof}
	Let $w=f_d\cdots f_1$ be the word corresponding to a geodesic $P$ from $v\in\Znm$ to $0$.
	Assume to the contrary that $P$ has no walls.
	By the Shift Direction Lemma \ref{lem:shift_direction_lemma_Znm}, exactly one of $\{\overleftarrow{s}_i, \overrightarrow{s}_i\}$ appears in $w$ for every $0\leq i\leq m$.
	Both $\overleftarrow{s}_0$ and $\overrightarrow{s}_m$ appear in $w$, since it has no outer walls. 
	Therefore both $\overleftarrow{s}_{i-1}$ and $\overrightarrow{s}_{i}$ appear in $w$ where $1\leq i\leq m$ is minimal such that $\overrightarrow{s}_i$ is in $w$.
	A contradiction to $(f_d\cdots f_1(v))_i=0$.
\end{proof}

\begin{fact}\label{fac:split_sum_center}
	Let $m,p,q$ be integers such that $0\leq p\leq m$ and $0\leq q\leq m$ and $|p-\frac{m}{2}| < |q-\frac{m}{2}|$. Then
	$\sum_{i=0}^{p}i + \sum_{i=0}^{m-p}i < \sum_{i=0}^{q}i + \sum_{i=0}^{m-q}i$.
\end{fact}

\begin{observation}\label{obs:corner_case_neqo}
	If $n=1$, then $\piv(v)=\{-1,\ldots,m+1\}$ for every $v\in\Znm$.
	Therefore, by Observation \ref{obs:trivial_path} and Fact \ref{fac:split_sum_center}, $\ecc_{\Znm[1][m]}(0)\leq \sum_{i=1}^{\lceil \frac{m}{2}\rceil}i + \sum_{i=1}^{\lfloor \frac{m}{2}\rfloor}i$.
	Let $v\in\Znm[1][m]$ such that $v_i=1$ for every $i$ with $1\leq i\leq m$.
	Clearly, $\ps_{-1}(v)=\ps_{m+1}(v)=\binom{m+1}{2}$ and $\ps_{p}(v)=\sum_{i=1}^{p}i + \sum_{i=1}^{m-p}i$ for every $0\leq p\leq m$. Therefore, by Lemma \ref{lem:geodesic_is_a_pivot_path} and Fact \ref{fac:split_sum_center}, $d(v,0)=\sum_{i=1}^{\lceil \frac{m}{2}\rceil}i + \sum_{i=1}^{\lfloor \frac{m}{2}\rfloor}i$. This proves that $\ecc_{\Znm[1][m]}(0)= \binom{\lceil \frac{m}{2}\rceil + 1}{2} + \binom{\lfloor \frac{m}{2}\rfloor + 1}{2}$, in accordance with Theorem \ref{thm:yoke_graph_diam} for the case $1=n\leq m$.
	In the rest of this paper we consider only dYoke graphs in which $n>1$.
\end{observation}

\begin{definition}[\sc $P$-interval]
	Let $P$ be a pivot path of $v$ and let $\{p_1,\ldots,p_t\}$ with $-1<p_1<\ldots<p_t<m+1$ be the set of inner walls of $P$.
	Denote $p_0 = -1$, $p_{t+1} = m+1$ and $Piv(P) = \{p_0, p_1, \ldots, p_t, p_{t+1}\}$.
	Note that $\piv(P)\subseteq \piv(v)$.
	We call every $p\in\piv(P)$ a \textbf{pivot of $P$}.
	$\piv(P)$ induces $t+1$ intervals $[p_k + 1,p_{k+1}]$ for $0\leq k\leq t$.
	Note that these intervals are a partition of $[0,m+1]$.
	We call every such interval a \textbf{$P$-interval}.
	For convenience, we say that an entry $v_i$ is in the interval $I$ if $i\in I$.
\end{definition}

\begin{lemma}\label{lem:same_direction_in_interval}
	Let $f_d\cdots f_1$ be the word corresponding to a pivot path $P$ of $v$ and let $I$ be a $P$-interval.
	Then for every $[i,i+1]\subseteq I$ and $[j,j+1]\subseteq I$, the instances of $s_i$ and $s_j$ in $f_d\cdots f_1$ shift in the same direction.
\end{lemma}
\begin{proof}
	By the Pivot Shift Direction Lemma \ref{lem:pivot_shift_direction_lemma}, exactly one of $\{\overleftarrow{s}_i, \overrightarrow{s}_i\}$ appears in $f_d\cdots f_1$, for every $[i,i+1]\subseteq I$.
	Assume to the contrary, without loss of generality, that both $\overleftarrow{s}_i$ and $\overrightarrow{s}_j$, appear in $f_d\cdots f_1$ for $i<j$ in $I$.
	Let $k\in I$ such that $i<k$ and $k$ is minimal such that $s_k$ appears shifting right in $f_d\cdots f_1$.
	Therefore both $\overleftarrow{s}_{k-1}$ and $\overrightarrow{s}_{k}$ appear in $f_d\cdots f_1$.
	Note that $1\leq k\leq m$.
	A contradiction to $(f_d\cdots f_1(v))_k=0$.
\end{proof}

\begin{definition}[\sc $d_I(P)$]
	By Lemma \ref{lem:same_direction_in_interval}, all of the unit shifts between entries in the same $P$-interval $I$ are in the same direction throughout $P$. 
	If the direction of these shifts is left (right), we say that \textbf{$P$ shifts left (right) in $I$}.
	Denote the number of unit shifts between entries in some interval $I\subseteq [0,m+1]$ within a pivot path $P$ by $d_I(P)$.
\end{definition}

\begin{lemma}\label{lem:number_of_shift_in_a_p_interval}
	Let $I=[p_1+1,p_2]$ be a $P$-interval of some pivot path $P$ of $v\in\Znm$.
	If $I$ contains one of the buckets, assume also that no step in $P$ shifts a unit from an empty bucket (i.e. $\overleftarrow{s}_m((\dots,0))=((\dots,n-1))$ and $\overrightarrow{s}_0((0,\dots))=((n-1,\dots))$ do not appear in $P$). Then
	
	\begin{enumerate}
		\item If $P$ shifts left in $I$, then $d_I(P) = \sum_{i\in I}iv_i$.
		\item If $P$ shifts right in $I$, then $d_I(P) = \sum_{i\in I}(m+1-i)v_i$.
	\end{enumerate}
\end{lemma}
\begin{proof}
	Note that by Lemma \ref{lem:same_direction_in_interval}, $P$ shifts either right or left in $I$.
	Assume that $P$ shifts left in $I$ (the proof of the second case follows by symmetric arguments).
	Note that every left unit shift in $P$ between entries in $I$ reduces the value of the sum $\sum_{i\in I}iv_i$ exactly by $1$, since a unit is never shifted left from an empty right bucket.
	Therefore $\sum_{i\in I}iv_i\geq 0$ and $d_I(P) = \sum_{i\in I}iv_i$.
\end{proof}

\begin{lemma}\label{lem:sufficient_for_number_of_shifts}
	Let $I=[p_1+1,p_2]$ be a $P$-interval of some pivot path $P$ of $v\in\Znm$.
	\begin{enumerate}
		\item Assume that $p_2=m+1$ and that $P$ shifts left in $I$. If $\sum_{i=j}^{m+1}v_i\geq 0$ for every $j\in I$, then there is no left unit shift in $P$ from an empty right bucket.
		\item Assume that $p_1=-1$ and that $P$ shifts right in $I$. If $\sum_{i=0}^{j}v_i\geq 0$ for every $j\in I$, then there is no right unit shift in $P$ from an empty left bucket.
	\end{enumerate}
\end{lemma}
\begin{proof}
	We prove the first case and similar arguments apply to the second case.
	Note that a left unit shift from an empty right bucket increases $\sum_{i\in I}iv_i$. 
	On the other hand, every other left unit shift between entries in $I$ (we call such unit shifts \textit{simple left unit shifts} in the rest of this proof) reduces the value of the sum $\sum_{i\in I}iv_i$ exactly by $1$.
	Therefore, it is sufficient to prove the following.
	\begin{claim}
		It is possible to set to $0$ every $v_i$ in $I$ using only simple left unit shifts between entries in $I$.
	\end{claim}
	Assume that $I$ contains no negative entries.
	If $p_1=-1$, then the claim is trivial.
	If $p_1$ is an inner pivot, then $v_i=0$ for every $i\in I$, since, in this case, a left unit shift between entries in $I$ cannot decrease the sum $\sum_{i\in I}v_i$ (implying, in fact, that $v_{m+1}=0$ and $|I|=1$).

	Otherwise, let $j$ be maximal in $I$ such that $v_j=-1$.
	The assumption that $\sum_{i=j}^{m+1}v_i\geq 0$ implies that $\sum_{i=j+1}^{m+1}v_i > 0$.
	Therefore, it is possible to set $v_j$ to $0$ using only simple left unit shifts between entries in $[j,m+1]$ without creating new negative entries in $I$.
	Note that $\sum_{i=j}^{m+1}v_i$ does not change in this process, since all of the unit shifts are simple.
	Therefore, this can be repeated until $I$ contains no negative entries.
	This proves the claim.
\end{proof}

\begin{lemma}\label{lem:no_pivots}
	If $\piv(v)=\{-1,m+1\}$ for some $v\in\Znm$, then $\ps_{m+1}(v)=\sum_{i=0}^{m+1}iv_i$, $\ps_{-1}(v)=\sum_{i=0}^{m+1}(m+1-i)v_i$ and $d(v,0)\leq \lfloor\frac{n(m+1)}{2}\rfloor$.
\end{lemma}
\begin{proof}
	Note that by Convention \ref{conv:cosets_are_integers}, $0<v_0,v_{m+1}<n$ and $\sum_{i=0}^{m+1}v_i = n$ since $\piv(v)=\{-1,m+1\}$.
	The assumption $\piv(v)=\{-1,m+1\}$ also implies that both $\sum_{i=0}^{j}v_i> 0$ and $\sum_{i=j}^{m+1}v_i> 0$ for every $0\leq j\leq m+1$.
	Therefore $\ps_{m+1}(v)=\sum_{i=0}^{m+1}iv_i$ and $\ps_{-1}(v)=\sum_{i=0}^{m+1}(m+1-i)v_i$ by Lemmas \ref{lem:number_of_shift_in_a_p_interval} and \ref{lem:sufficient_for_number_of_shifts}.
	
	Clearly $\sum_{i=0}^{m+1}iv_i + \sum_{i=0}^{m+1}(m+1-i)v_i= (m+1) \sum_{i=0}^{m+1} v_i =n(m+1)$.
	Therefore, by Lemma \ref{lem:geodesic_is_a_pivot_path}, 
	$d(v,0)=\min\{\ps_{-1}(v),\ps_{m+1}(v)\}\leq \lfloor\frac{\ps_{-1}(v)+\ps_{m+1}(v)}{2}\rfloor=\lfloor\frac{n(m+1)}{2}\rfloor.$

\end{proof}

\begin{lemma}\label{lem:ecc_lower_bound}
	For every dYoke graph $\Znm$ we have $\ecc_{\Znm}(0)\geq\lfloor\frac{n(m+1)}{2} \rfloor$.
\end{lemma}
\begin{proof}
	Let $u\in\Znm$ be defined as follows:
	$u_0=u_{m+1}=\lfloor\frac{n}{2}\rfloor$. 
	If $n$ is even, then $u_i=0$ for all $1\leq i\leq m$ and if $n$ is odd, then $u_{\lfloor\frac{m+1}{2}\rfloor}=1$ and $u_i=0$ for all other $1\leq i\leq m$.
	In all of these cases, $\piv(u)=\{-1,m+1\}$.
	Therefore, by Lemma \ref{lem:no_pivots}, $\ps_{m+1}(u)=\sum_{i=0}^{m+1}i u_i=\lfloor\frac{n(m+1)}{2}\rfloor$ and $\ps_{-1}(u) =\sum_{i=0}^{m+1}(m+1-i) u_i=\lceil\frac{n(m+1)}{2}\rceil$.
	It follows that $\ecc_{\Znm}(0)\geq d(u,0)=\min\{\ps_{m+1}(u), \ps_{-1}(u)\}=\lfloor\frac{n(m+1)}{2} \rfloor$.
\end{proof}

\begin{theorem}\label{thm:ecc_n_geq_m}
	If $m\leq n$, then $\ecc_{\Znm}(0) = \lfloor\frac{n(m+1)}{2}\rfloor$.
\end{theorem}
\begin{proof}
	By Lemma \ref{lem:ecc_lower_bound}, it is sufficient to show that $\ecc_{\Znm}(0) \leq \lfloor\frac{n(m+1)}{2}\rfloor$.
	Let $v\in\Znm$. 
	By Lemma \ref{lem:no_pivots}, we can assume that $v$ has some inner pivot $p\in\piv(v)$.
	By Observation \ref{obs:trivial_path} and Fact \ref{fac:split_sum_center}, $d(v,0) \leq \sum_{i=0}^{p}i  + \sum_{i=0}^{m-p}i \leq \binom{m+1}{2}$.
	Since $m \le n$, $\binom{m+1}{n} \le \frac{n(m+1)}{2}$ and therefore $\binom{m+1}{2}\leq\lfloor\frac{n(m+1)}{2}\rfloor$.
\end{proof}

\subsection{The Case $n\leq m$}\label{subsec:n_leq_m}
In this subsection we calculate the eccentricity of $0$ in $\Znm$ in the case $n\leq m$.
Throughout this subsection we assume that $2\leq n\leq m$.
For $v\in \Znm$ we introduce the following notations:
\begin{enumerate}
	\item $p_l(v)=\max\{p\in \piv(v):p<\frac{m}{2}\}$.
	\item $p_r(v)=\min\{p\in \piv(v):p\geq\frac{m}{2}\}$.
	\item $I_c(v) = [p_l(v)+1,p_r(v)]$.
	\item $h(v)=\min\{|p-\frac{m}{2}|:p\in \piv(v)\}$.
	\item $\hnm = 
	\begin{cases}
	\frac{n}{2} & \mbox{if } 2\divides (m-n) \\
	\frac{n+1}{2} & \mbox{if } 2\ndivides (m-n)
	\end{cases}.$
	
\end{enumerate}
We abbreviate $p_l$, $p_r$ and $I_c$ when $v$ is evident.

\begin{definition}\label{def:aut_mu}
	Let $\mu:\Znm\rightarrow \Znm$ be the automorphism of $\Znm$ defined by $\mu(v)_i=-v_i$ for every $0\leq i\leq m+1$.
\end{definition}
\begin{observation}\label{obs:neg_Ic_to_positive}
	Let $v\in\Znm$ such that $\sum_{i\in I_c}v_i=-n$.
	Recall that, by Convention \ref{conv:cosets_are_integers},  the buckets are identified with non-negative integers.
	Therefore $I_c\subseteq [1,m]$ and $\sum_{i\in I_c}\mu(v)_i=n$.
\end{observation}

The outline of this subsection is as follows.
We first construct two candidates for an antipode of $0$ in Definitions \ref{def:uz} and \ref{def:uo} and compute their distance from $0$ in Lemma \ref{lem:dist_of_uz}.
Then we prove that the maximal value of the two distances, is an upper bound on the distance of an arbitrary vertex $v$ in $\Znm$ from $0$.
We split this proof into three cases.
Note that by Observation \ref{obs:neg_Ic_to_positive}, we can assume that $\sum_{i\in I_c}v_i\in\{0,n\}$, since $0$ is a fixed point of $\mu$.
\begin{enumerate}
	\item $h(v)<\hnm$ (Lemma \ref{lem:close_to_middle_pivot}).
	\item $h(v)\geq\hnm$ and $\sum_{i\in I_c(v)}v_i=n$ (Lemma \ref{lem:middle_interval_is_n}).
	\item $h(v)\geq\hnm$ and $\sum_{i\in I_c(v)}v_i=0$ (Lemma \ref{lem:zero_middle_interval}).
\end{enumerate}

\begin{definition}[\sc $\uz$]\label{def:uz}
	Let $u\in\Znm$ such that $u_i=1$ for all $i$ with \mbox{$1\leq i\leq m$} and $u_0\equiv -\lfloor\frac{m-n}{2} \rfloor \bmod n$. 
	Denote this vertex by $\uz$ and denote $\dz=d(\uz,0)$.
\end{definition}
For example, $\uz[3][5]=(2,1,1,1,1,1,2)$ and $\uz[3][6]=(2,1,1,1,1,1,1,1)$.

\begin{definition}[\sc $\uo$]\label{def:uo}
	Let $n$ and $m$ be two integers such that $2\leq n<m$ and $2\ndivides (m-n)$.
	Let $u\in\Znm$ such that $u_{i}=0$ for $i=\lceil\frac{m+1}{2}\rceil$, $u_i=1$ for all $i$ with $1\leq i\leq m$ and $u_0\equiv -\lfloor\frac{m-n}{2} \rfloor \bmod n$. 
	Denote this vertex by $\uo$ and denote $\doo=d(\uo,0)$.
\end{definition}
For example, $\uo[2][5]=(1,1,1,0,1,1,1)$ and $\uo[3][6]=(2,1,1,1,0,1,1,2)$.

\begin{lemma}\label{lem:dist_of_uz}\ 
	\begin{enumerate}
		\item $\dz = \binom{\lfloor\frac{m+n}{2}\rfloor+1}{2} + \binom{\lceil\frac{m-n}{2}\rceil+1}{2}$.
		\item $\doo=\binom{\lceil\frac{m+n}{2}\rceil+1}{2} + \binom{\lfloor\frac{m-n}{2}\rfloor+1}{2} - \lceil\frac{m+1}{2}\rceil$ (for $n<m$ such that $2\ndivides (m-n)$).
	\end{enumerate}
	
\end{lemma}
\begin{proof}
\begin{enumerate}
\item By Lemma \ref{lem:number_of_shift_in_a_p_interval}, $\ps_{p}(\uz)=\sum_{i=0}^{p}i+\sum_{i=0}^{m-p}i=\binom{p+1}{2} + \binom{m-p+1}{2}$ for every inner pivot $p$ of $\uz$.
Note that $p_l(\uz)=\lfloor\frac{m-n}{2} \rfloor$ and $p_r(\uz)=\lfloor\frac{m+n}{2} \rfloor$, implying that $|p_r-\frac{m}{2}|\leq |p_l-\frac{m}{2}|$.
Note also, that the length of an outer pivot path is at least $\binom{m+1}{2}$.
Therefore, and by Lemma \ref{lem:geodesic_is_a_pivot_path} and Fact \ref{fac:split_sum_center}, $\dz=\ps_{p_r}(\uz)=\binom{p_r+1}{2} + \binom{m-p_r+1}{2}=\binom{\lfloor\frac{m+n}{2}\rfloor+1}{2} + \binom{\lceil\frac{m-n}{2}\rceil+1}{2}$.

\item Note that $p_l(\uo)=\lfloor\frac{m-n}{2}\rfloor$ and $p_r(\uo)=\lceil\frac{m+n}{2}\rceil$.
Therefore, in this case, $|p_r-\frac{m}{2}|=|p_l-\frac{m}{2}|$ and by Definition \ref{def:uo}, $\ps_{p_r}(\uo)\leq \ps_{p_l}(\uo)$, since the entry equal to $0$ in $I_c$ is indexed by $\lceil\frac{m+1}{2}\rceil$. This implies that $\doo=\ps_{p_r}(\uo)=\binom{\lceil\frac{m+n}{2}\rceil+1}{2} + \binom{\lfloor\frac{m-n}{2}\rfloor+1}{2} - \lceil\frac{m+1}{2}\rceil$.
\end{enumerate}
\end{proof}

\begin{definition}[\sc $\dnm$]
	Denote 
	$$
	\dnm = 
	\begin{cases}
	\dz & \mbox{if } 2\divides (m-n) \\
	\max\{\dz, \doo\} & \mbox{if } 2\ndivides (m-n)
	\end{cases}.
	$$
\end{definition}

\begin{lemma}\label{lem:close_to_middle_pivot}
	Let $v\in\Znm$ such that $h(v)<\hnm$. Then $d(v,0)\leq \dz\leq\dnm$.
\end{lemma}
\begin{proof}
	Let $p'$ be an inner pivot of $v$ for which $|p'-\frac{m}{2}|=h(v)$.
	By Observation \ref{obs:trivial_path}, $d(v,0)\leq \sum_{i=1}^{p'}i+\sum_{i=1}^{m-p'}i$.
	Note that if $2\divides(m-n)$, then $h(\uz)=\hnm$, and if $2\ndivides(m-n)$, then $h(\uz)=\hnm-1$. 
	Therefore $h(v)\leq h(\uz)$.
	By the proof of Lemma \ref{lem:dist_of_uz}, $\dz = \sum_{i=1}^{p}i+\sum_{i=1}^{m-p}i$ for $p\in\piv(\uz)$ for which $|p-\frac{m}{2}|=h(\uz)$.
	Therefore, by Fact \ref{fac:split_sum_center}, $d(v,0)\leq \dz$, since $|p'-\frac{m}{2}|\leq |p-\frac{m}{2}|$.
\end{proof}

The following lemma generalizes the part in Lemma \ref{lem:no_pivots} which states that if $\piv(v)=\{-1,m+1\}$ for some $v\in\Znm$, then $d(v,0)\leq \lfloor\frac{n(m+1)}{2}\rfloor$.

\begin{lemma}\label{lem:v_hat_dist_from_zero_center}
	Let $v\in\Znm$ such that $\sum_{i\in I_c}v_i=n$.
	Then $d(v,0)\leq \sum_{i=1}^{p_l}i + \sum_{i=1}^{m-p_r}i + \lfloor\frac{n(m+1)}{2}\rfloor.$
\end{lemma}
\begin{proof}
	$\sum_{i=1}^{p_l}i + \sum_{i=1}^{m-p_r}i$ is an upper bound on the number of steps needed to set to $0$ every $v_i$ not in $I_c$.
	We can therefore assume that $v_i=0$ for every $i\notin I_c$ and prove that $d(v,0)\leq \lfloor\frac{n(m+1)}{2}\rfloor$.
	Since $\sum_{i\in I_c}v_i=n$, every $p_r$-pivot path of $v$ shifts left in $[0,p_r]$ and every $p_l$-pivot path of $v$ shifts right in $[p_l+1,m+1]$.
	Moreover, the assumption $\sum_{i\in I_c}v_i=n$ implies that $\sum_{i=j}^{p_r}v_i> 0$ for every $j\in I_c$ and $\sum_{i=p_l+1}^{j}v_i> 0$ for every $j\in I_c$.
	Therefore, by Lemmas \ref{lem:number_of_shift_in_a_p_interval} and \ref{lem:sufficient_for_number_of_shifts}, $d(v,0)\leq\min\{\ps_{p_l}(v),\ps_{p_r}(v)\}\leq \lfloor\frac{\sum_{i\in I_c}iv_i + \sum_{i\in I_c}(m+1-i)v_i}{2}\rfloor=\lfloor\frac{n(m+1)}{2}\rfloor.$
\end{proof}

\begin{observation}\label{obs:dz_doo_split}
If $2|(m-n)$ and $v=\uz$, then $\sum_{i\in I_c}iv_i = \sum_{i\in I_c}(m+1-i)v_i$.
Moreover, if $v=\uo$, then $|\sum_{i\in I_c}iv_i - \sum_{i\in I_c}(m+1-i)v_i|\leq 1$.
Therefore, the following is implied by the proofs of Lemmas \ref{lem:dist_of_uz} and \ref{lem:v_hat_dist_from_zero_center}.	
	\begin{enumerate}
		\item If $2|(m-n)$, then $\dz = \sum_{i=1}^{\frac{m-n}{2}}i + \sum_{i=1}^{m-\frac{m+n}{2}}i + \frac{n(m+1)}{2}$.
		\item $\doo = \sum_{i=1}^{\lfloor\frac{m-n}{2}\rfloor}i + \sum_{i=1}^{m-\lceil\frac{m+n}{2}\rceil}i + \lfloor\frac{n(m+1)}{2}\rfloor$.
	\end{enumerate}
\end{observation}

\begin{lemma}\label{lem:middle_interval_is_n}
	Let $v\in\Znm$ such that $h(v)\geq\hnm$ and $\sum_{i\in I_c(v)}v_i=n$. 
	Then $d(v,0) \leq \dnm$.
\end{lemma}
\begin{proof}
	Assume that $2|(m-n)$.
	By Observation \ref{obs:dz_doo_split}, 
	$\dnm = \dz = \sum_{i=1}^{\frac{m-n}{2}}i + \sum_{i=1}^{m-\frac{m+n}{2}}i + \frac{n(m+1)}{2}.$
	By Lemma \ref{lem:v_hat_dist_from_zero_center}, $d({v},0)\leq \sum_{i=1}^{p_l}i + \sum_{i=1}^{m-p_r}i + \frac{n(m+1)}{2}.$
	Note that both $p_l\leq\frac{m-n}{2}$ and $p_r\geq\frac{m+n}{2}$, since $h(v)\geq\hnm$.
	Therefore
	$ \dnm - d(v,0)\geq \sum_{i=p_l+1}^{\frac{m-n}{2}}i + \sum_{i=\frac{m+n}{2}+1}^{p_r}m+1-i \geq 0$. 
	
	Otherwise assume that $2\ndivides(m-n)$.
	By Observation \ref{obs:dz_doo_split}, 
	$\dnm\geq \doo = \sum_{i=1}^{\lfloor\frac{m-n}{2}\rfloor}i + \sum_{i=1}^{m-\lceil\frac{m+n}{2}\rceil}i + \lfloor\frac{n(m+1)}{2}\rfloor.$
	By Lemma \ref{lem:v_hat_dist_from_zero_center}, $d({v},0)\leq \sum_{i=1}^{p_l}i + \sum_{i=1}^{m-p_r}i + \lfloor\frac{n(m+1)}{2}\rfloor.$
	Note that both $p_l\leq \lfloor\frac{m-n}{2}\rfloor$ and $p_r\geq\lceil\frac{m+n}{2}\rceil$, since $h(v)\geq\hnm$.
	Therefore
	$
	\dnm - d(v,0) \geq
	\sum_{i=p_l+1}^{\lfloor\frac{m-n}{2}\rfloor}i + \sum_{i=\lceil\frac{m+n}{2}\rceil+1}^{p_r}m+1-i \geq 0
	$.
\end{proof}

\begin{lemma}\label{lem:upper_middle_is_zero}
	Let $v\in\Znm$ such that $I_c(v)\subseteq[1,m]$ and $\sum_{i\in I_c(v)}v_i=0$. Then 
	$$d(v,0) \leq \sum_{i=1}^{p_l}i + \sum_{i=1}^{m-p_r}i + \lfloor\frac{\Icwidth}{2}\rfloor\lceil\frac{\Icwidth}{2}\rceil$$ 
	where $\Icwidth=|I_c(v)|$.
\end{lemma}
\begin{proof}
	As in the beginning of the proof of Lemma \ref{lem:v_hat_dist_from_zero_center}, we can assume that $v_i=0$ for every $i\notin I_c$ and prove that $d(v,0)\leq \lfloor\frac{\Icwidth}{2}\rfloor\lceil\frac{\Icwidth}{2}\rceil$.
	
	Let $P$ be any $p_l$- or $p_r$-pivot path of $v$.
	$I_c$ is clearly a $P$-interval, since $\sum_{i\in I_c}v_i=0$ and $I_c(v)\subseteq[1,m]$.
	Assume without loss of generality, that $P$ shifts left in $I_c$.
	By Lemma \ref{lem:number_of_shift_in_a_p_interval}, $d_{I_c}(P) = \sum_{i\in I_c}iv_i$.
	The number of entries of $v$ that are equal to $1$ is equal to the number of entries of $v$ that are equal to $-1$, since $\sum_{i\in I_c}v_i = 0$.
	$I_c$ contains at most $\lfloor\frac{\Icwidth}{2}\rfloor$ entries equal to $1$ and the same number of entries equal to $-1$. The maximum value of $\sum_{i\in I_c}iv_i$ is obtained when the first $\lfloor\frac{\Icwidth}{2}\rfloor$ entries in $I_c$ are equal to $-1$ and the last $\lfloor\frac{\Icwidth}{2}\rfloor$ entries are equal to $1$.
	Therefore $\sum_{i\in I_c}iv_i\leq \lfloor\frac{\Icwidth}{2}\rfloor\lceil\frac{\Icwidth}{2}\rceil$.
\end{proof}

\begin{fact}\label{fac:split_sum}
	Let $a$ and $b$ be two positive integers.
	Then $\sum_{i=1}^a i + \sum_{i=1}^{b} i > \lfloor\frac{a+b}{2}\rfloor\lceil\frac{a+b}{2}\rceil$.
\end{fact}

\begin{lemma}\label{lem:zero_middle_interval}
	Let $v\in\Znm$ such that $h(v)\geq\hnm$ and $\sum_{i\in I_v}v_i=0$.
	Then $d(v,0)\leq\dnm$.
\end{lemma}
\begin{proof}
	If both $p_l$ and $p_r$ are outer pivots, then both buckets are non-empty and $\sum_{i\in I_v}v_i=n$, since, in this case, $v$ has no pivots other than $p_l$ and $p_r$. Contradicting the assumption $\sum_{i\in I_v}v_i=0$.
	Assume that one of $\{p_l,p_r\}$, say $p_l$, is an outer pivot (and that the other pivot, $p_r$, is an inner pivot).
	Let $\mu$ be the automorphism of $\Znm$ in Definition \ref{def:aut_mu}.
	Note that by Convention \ref{conv:cosets_are_integers}, if a bucket of $v\in\Znm$ is non-empty (and positive), then it is also non-empty (and positive) in $\mu(v)$.
	Moreover, $I_c(v)=I_c(\mu(v))$.
	This implies that $\sum_{i\in I_c(\mu(v))}\mu(v)_i=n$.
	Therefore, since $0$ is a fixed point of $\mu$, by Lemma \ref{lem:middle_interval_is_n}, $d(v,0)=d(\mu(v),0)\leq \dnm$ as required.
	
	Otherwise $I_c(v)\subseteq[1,m]$.
	By Lemma \ref{lem:upper_middle_is_zero}, 
	$d({v},0) \leq 
	\sum_{i=1}^{p_l}i + \sum_{i=1}^{m-p_r}i + \lfloor\frac{\Icwidth}{2}\rfloor\lceil\frac{\Icwidth}{2}\rceil$ where $\Icwidth=|I_c(v)|$.
	Let $p=\lfloor\frac{m+n}{2}\rfloor$ so that $\dz=\sum_{i=1}^p i + \sum_{i=1}^{m-p} i$.
	Note that $|p-\frac{m}{2}|\leq \hnm$.
	By assumption, $p_l\leq p\leq p_r$. Therefore
	\begin{align*}
	\dnm - d(v,0) \geq
	\dz - d(v,0) \geq 
	\sum_{i=p_l+1}^p i + \sum_{i=m-p_r+1}^{m-p} i - \lfloor\frac{\Icwidth}{2}\rfloor\lceil\frac{\Icwidth}{2}\rceil.
	\end{align*}
	$\sum_{i=p_l+1}^p i + \sum_{i=m-p_r+1}^{m-p} i \geq \sum_{i=1}^{p-p_l} i + \sum_{i=1}^{p_r-p} i$.
	Let $a=p-p_l$, let $b=p_r-p$ and note that $a+b=\Icwidth$. Therefore, by Fact \ref{fac:split_sum},
	$
	\dnm - d(v,0) \geq \sum_{i=1}^a i + \sum_{i=1}^{b} i - \lfloor\frac{\Icwidth}{2}\rfloor\lceil\frac{\Icwidth}{2}\rceil>0.
	$
\end{proof}

\begin{theorem}\label{thm:ecc_n_leq_m}
	Let $n, m$ be two integers where $2\leq n\leq m$. 
	If $2\divides (m-n)$ or $n\leq\lceil\frac{m+1}{2}\rceil$, then 
	$$\ecc_{\Znm}(0) = \dz = \binom{\lfloor\frac{m+n}{2}\rfloor+1}{2} + \binom{\lceil\frac{m-n}{2}\rceil+1}{2}.$$
	Otherwise, 
	$$\ecc_{\Znm}(0) = \doo = \dz + n-\lceil\frac{m+1}{2}\rceil.$$
\end{theorem}
\begin{proof}
	Note that if $2\ndivides (m-n)$, then, by Lemma \ref{lem:dist_of_uz}, $\doo - \dz = n-\lceil\frac{m+1}{2}\rceil.$
	Therefore, the theorem merely states that $\ecc_{\Znm}(0)=\dnm$ when $n\leq m$.
	By Lemma \ref{lem:dist_of_uz}, $\ecc_{\Znm}(0)\geq \dnm$.
	
	Let $v\in \Znm$.
	By Observation \ref{obs:neg_Ic_to_positive}, we can assume that $\sum_{i\in I_c}v_i\in\{0, n\}$, since $0$ is a fixed point of the automorphism $\mu$ in Definition \ref{def:aut_mu}.
	Therefore, by Lemmas \ref{lem:close_to_middle_pivot}, \ref{lem:middle_interval_is_n} and \ref{lem:zero_middle_interval}, $d(v,0)\leq \dnm$ and $\ecc_{\Znm}(0)\leq \dnm$.
\end{proof}

At this point, we can combine our results to prove the main theorem of this paper.

\begin{proof}[\sc Proof of Theorem \ref{thm:yoke_graph_diam}]
	By Theorem \ref{thm:ecc_eq_diam}, for every Yoke graph $\Ynm$, the diameter of $\Ynm$ is equal to the eccentricity of $0$ in the corresponding dYoke graph $\Znm$.
	
	In Observations \ref{obs:corner_case_meqz}, \ref{obs:corner_case_neqo} and Theorems \ref{thm:ecc_n_geq_m}, \ref{thm:ecc_n_leq_m}, the eccentricity of $0$ in $\Znm$ is shown to be equal to the value of the diameter of $\Ynm$ as stated in this theorem.
\end{proof}

\section{Additional Problems}\label{sec:last_section}
A simpler version of the calculation of $\ecc_{\Znm}(0)$ can be used to calculate $\ecc_{\Ynm}(0)$ and show that $\ecc_{\Ynm}(0)=\ecc_{\Znm}(0)$. 
Therefore a direct proof that $0$ is an antipode in $\Ynm$ might provide a shorter proof for the diameter of $\Ynm$.
\begin{question}
	Is it possible to show, without using $\Znm$, that $0$ is an antipode in $\Ynm$?
\end{question}

It stands out that the three examples of flip graphs that motivated the definition of Yoke graphs and this paper, are all Yoke graphs that belong to the case $m<n$.

\begin{question}
	Are there any "interesting" examples of Yoke graphs for the case $n<m$?
\end{question}

\section{Acknowledgments}
This work is part of the author's Ph.D. thesis written under the supervision of Prof. Ron M. Adin and Prof. Yuval Roichman. The author would like to thank Dr. Luie Jennings, Dr. Menachem Shlossberg and Dr. Arnon Netzer for many helpful discussions and suggestions.



\begin{thebibliography}{ABR4}
	\bibitem{TFT1} R.M. Adin, M. Firer and Y. Roichman. {\it Triangle-free triangulations}. Adv. in Appl. Math. 45 (2010), 77--95.
	\bibitem{bose} P. Bose and F. Hurtado, {\it Flips in planar graphs}. Comput.\ Geom.\ 42 (2009), 60--80.	
	\bibitem{elizalde} S. Elizalde and Y. Roichman, {\it Arc permutations}, J. Algebraic Combinat. 39 (2014), 301--334.
	\bibitem{fabila} R. Fabila-Monroy, D. Flores-Penaloza, C. Huemer, F. Hurtado, D.R. Wood and J. Urrutia. {\it On the chromatic number of some flip graphs}. Discrete Mathematics and Theoretical Computer Science 11 (2009), 47--56.
	\bibitem{jennings_gpm} R.H. Jennings, {\it Geodesics in a Graph of Perfect Matchings}, S\'{e}minaire Lotharingien de Combinatoire 74 (2017): B74e.
	\bibitem{jennings} R.H. Jennings, Ph.D. Thesis, Bar-Ilan University, in preparation.
	\bibitem{keller} C. Keller and M.A. Perles, {\it On the smallest sets blocking simple perfect matchings in a convex geometric graph}. Israel J.\ Math.\ 187 (2012), 465--484.
	\bibitem{perles} C. Keller and M.A. Perles, {\it Characterization of co-blockers for simple perfect matchings in a convex geometric graph}. Discrete Comput. Geom. 50 (2013), 491--502.
	\bibitem{yuval} Y. Khachatryan, {\it Words, Groups and Graphs}. M.Sc. Thesis, Bar-Ilan University, 2015.
	\bibitem{martin} J.L. Martin, M. Morin, and J.D. Wagner, {\it On distinguishing trees by their chromatic symmetric functions}. J. Combin. Theory Ser. A 115 (2008), 237--253.
	\bibitem{parlier} H. Parlier and S. Zappa, {\it Distances in domino flip graphs}. Amer.\ Math.\ Monthly 124-8 (2017), 710--722.
	\bibitem{pournin} L. Pournin. {\it The diameter of associahedra}. Adv.\ in Math.\ 259 (2014), 13--42.	
	\bibitem{sage} The Sage Developers, {\it {S}age {M}athematics {S}oftware {S}ystem}, {S}ageMath. 2017, {\tt http://www.sagemath.org}.
	\bibitem{tarjan} D.D. Sleator, R.E. Tarjan and W.P. Thurston, {\it Rotation distance, triangulations, and hyperbolic geometry}, J. Amer. Math. Soc. 1 (1988), 647--681.
\end{thebibliography}
\end{document}